\documentclass[a4paper,11pt]{amsart}

\usepackage{amsmath}
\usepackage{amssymb}
\usepackage{amsthm}
\usepackage[all]{xy}

\newtheorem{thm}{Theorem}[section]
\newtheorem{lemma}{Lemma}[section]

\newtheorem{remark}{Remark}[section]

 \newcommand{\ux}{\underline{x}}

 \newcommand{\upx}{\partial_{\underline{x}}}

 \newcommand{\laf}{\mathfrak f}
\newcommand{\lag}{\mathfrak g}
\newcommand{\lap}{\mathfrak p}

\newcommand{\lan}{\mathfrak n}

\newcommand{\lasp}{\mathfrak{sp}}
\newcommand{\laosp}{\mathfrak{osp}}

\newcommand{\lagl}{\mathfrak{gl}}

\newcommand{\lat}{\mathfrak{t}}

\newcommand{\mS}{\mathbb S}

\newcommand{\mC}{\mathbb C}

\newcommand{\mR}{\mathbb R}
\newcommand{\mN}{\mathbb N}

\newcommand{\mZ}{\mathbb Z}

\newcommand{\SO}{{SO}}

\newcommand{\GL}{{GL}}

\newcommand{\Spin}{{Spin}}

 \newcommand{\cH}{\mathcal{H}}

\newcommand{\cC}{\mathcal{C}}
\newcommand{\cR}{\mathcal{R}}
\newcommand{\cD}{\mathcal{D}}
\newcommand{\cS}{\mathcal{S}}

\newcommand{\cP}{\mathcal{P}}
\newcommand{\cM}{\mathcal{M}}

\newcommand{\cA}{\mathcal A}

\newcommand{\cJ}{\mathcal J}
\newcommand{\cI}{\mathcal I}
\newcommand{\cU}{\mathcal U}
\newcommand{\cW}{\mathcal W}

\newcommand{\End}{\operatorname{End}}
\newcommand{\Ind}{\operatorname{Ind}}
\newcommand{\lspan}{\operatorname{span}}
\newcommand{\Ker}{\operatorname{Ker}}
\newcommand{\trace}{\operatorname{tr}}
\newcommand{\largewedge}{\mbox{\Large $\wedge$}}

\begin{document}

\title[Fischer decomposition in several variables]{Fischer decomposition for spinor valued polynomials in several variables}

\author{R. L\'avi\v cka}

\author{V. Sou\v cek}

\address{Charles University, Faculty of Mathematics and Physics, Mathematical Institute\\
Sokolovsk\'a 83, 186 75 Praha, Czech Republic}

\email[R. L\'avi\v cka]{lavicka@karlin.mff.cuni.cz}

\email[V. Sou\v cek]{soucek@karlin.mff.cuni.cz}

\thanks{The support of the grant GACR 17-01171S is gratefully acknowledged.}


\begin{abstract}
It is well-known that polynomials decompose into spherical harmonics. This result is called separation of variables or the Fischer decomposition.  
In the paper  we prove the Fischer decomposition for spinor valued polynomials in $k$ vector variables of $\mR^m$ under the stable range condition $m\geq 2k$. Here the role of spherical harmonics is played by monogenic polynomials, that is, polynomial solutions of the Dirac equation in $k$ vector variables. 
\end{abstract}


\maketitle


\section{Introduction}
 
Each polynomial $P$ in the Euclidean space $\mR^m$ decomposes uniquely as
$$P=H_0+r^2H_1+\cdots+r^{2j}H_j+\cdots$$
where $r^2=x_1^2+\cdots+x_m^2$, $(x_1,\ldots,x_m)\in\mR^m$ and $H_j$ are harmonic polynomials in $\mR^m$.
Under the natural action of the orthogonal group $\SO(m)$, this decomposition of polynomials is invariant, $r^2$ generates the algebra of invariant polynomials and the whole space of  polynomials is the tensor product of invariants and spherical harmonics. 
An analogous result known as separation of variables or the Fischer decomposition was obtained in various cases and 
for other symmetry groups \cite{CW,G,Ho1,KV,Cou,LS}.

For example, spinor valued polynomials in one variable of $\mR^m$ decompose into monogenic polynomials \cite{DSS}. 
A~polynomial $P$ in $\mR^m$ taking values in the spinor space $\mS$ is called monogenic if it satisfies the equation $\upx P=0$ where 
$$\upx:=\sum_{i=1}^m e_i\partial_{x_i}$$
is the Dirac operator in $\mR^m$. Here $(e_1,\ldots,e_m)$ is an orthonormal basis for $\mR^m$.
Then each polynomial $P:\mR^m\to\mS$ decomposes uniquely as
$$P=M_0+\ux M_1+\cdots+\ux^j M_j+\cdots$$
where $\ux=x_1e_1+\cdots+x_me_m$ is the vector variable of $\mR^m$ and $M_j$ are monogenic polynomials.

Function theory for the Dirac operator $\upx$ is called Clifford analysis and it generalizes complex analysis to higher dimensions.
An important feature of the Fischer decomposition
 is the fact that its pieces behave well under the action of the
 conformal spin group, which is the symmetry group of the Dirac equation.
Its transcendental version is usually called the Almansi
 decomposition. 
 
Clifford analysis in several variables is a natural generalization
of theory of functions of several complex variables. Its study started about half century ago, however its
basic principles and facts are still not well understood.
One of the most important basic questions is to formulate and to prove an analogue of the Fischer decomposition for spinor valued polynomials in $k$ vector variables of $\mR^m$, that is, for polynomials $P:(\mR^m)^k\to\mS$.

It became soon clear that it is necessary to distinguish different cases. The simplest case is the so called stable range $m\geq 2k$.  
After a longer evolution, a corresponding conjecture was formulated 
in the book by F.~Colombo, F.~Sommen, I.~Sabadini, D.~Struppa (\cite{CSSS}, Conj.\ 4.2.1, p.\ 236). Recently, the conjecture was proved in the case of two variables (\cite{L}). 

The main purpose of the paper is to prove the Fischer decomposition
for any number of variables in the stable range. The proof of the conjecture is based on the Fischer decomposition for scalar valued polynomials in several variables.
The scalar case in the stable range is well known (see \cite{GW}). Recently,  in the scalar case, an alternative approach leading to explicit formulae for different projections in the decomposition
 was developed in \cite{DER} for two variables.

Let us state the main result. 
Polynomials $P:(\mR^m)^k\to\mS$ depend on $k$ vector variables 
$\ux_1,\ldots,\ux_k$ of $\mR^m$ with $\ux_j=x_{1j}e_1+\cdots+x_{mj}e_m$.
Such a~polynomial $P$ is called monogenic if it satisfies all the corresponding Dirac equations 
$$\partial_{\ux_1} P=0,\ldots, \partial_{\ux_k} P=0.$$ It is easy to see that the Dirac operators $\partial_{\ux_j}$ and the left multiplication by the vector variables $\ux_j$ are invariant operators on the space $\cP((\mR^m)^k)\otimes\mS$ of spinor valued polynomials on $(\mR^m)^k$ under the natural action of the spin group $\Spin(m)$, the double cover of $\SO(m)$. Denote by $\cJ$ the algebra of invariants generated by the vector variables $\ux_1,\ldots,\ux_k$. Then, in the stable range, we prove that 
$\cP((\mR^m)^k)\otimes\mS$ is isomorphic to the tensor product $\cJ\otimes\cM$
where $\cM$ is the space of spinor valued monogenic polynomials on $(\mR^m)^k.$   
Indeed, we show the following result we refer to as the monogenic Fischer decomposition in the stable range.

\begin{thm}\label{t_mfd} For $J\subset\{1,2,\ldots,k\}$, $J=\{j_1<\cdots<j_r\}$, denote $$\ux_J=\ux_{j_1}\cdots \ux_{j_r}.$$
For $1\leq i\leq j\leq k$, put
$
r^2_{ij}:=x_{1i} x_{1j}+x_{2i} x_{2j}+\cdots+x_{mi} x_{mj}.
$

\smallskip\noindent
	If $m\geq 2k$, then we have
	\begin{equation}\label{e_mfd}
	\cP((\mR^m)^k)\otimes\mS=\bigoplus_{J,\{n_{ij}\}} \Big(\prod_{1\leq i\leq j\leq k}r_{ij}^{2n_{ij}}\Big) \ux_J  \cM
	\end{equation}
	where the direct sum is taken over all subsets $J$ of $\{1,2,\ldots,k\}$ and all sequences $\{n_{ij}|\ 1\leq i\leq j\leq k\}$ of numbers of $\mN_0$.
\end{thm}

For two variables, the stable range means that the dimension $m$ is bigger or equal to 4. In \cite{L}, it was proved that
the Fischer decomposition holds also for dimension 3. There are indications that, in the scalar case, the 'stable' Fischer decomposition  could hold also in one dimension less than the stable range in general, that is, when $m\geq 2k-1$. If so,
the same would be true also for the monogenic Fischer decomposition.
Out of the stable range, there is no reasonable conjecture available and
further study is needed.

In the space $\cP((\mR^m)^k)$ of scalar valued polynomials on $(\mR^m)^k$,
the isotypic components of the action of $\SO(m)$ 
have infinite multiplicities. The Howe duality theory removes the multiplicities and shows that the space
$\cP((\mR^m)^k)$ decomposes under the action of 
the dual pair $SO(m)\times\lasp(2k)$ with multiplicity one (\cite{Ho1,G,HTW}). An analogous role for spinor valued polynomials $\cP((\mR^m)^k)\otimes\mS$ is played by the pair $\Spin(m)\times\laosp(1|2k)$ where the Lie superalgebra $\laosp(1|2k)$ is generated by the odd operators $\partial_{\ux_j}$ and $\ux_j$ for $j=1,\ldots,k$. 
Indeed, we can reformulate Theorem \ref{t_mfd} as a~duality between finite dimensional representations of $\Spin(m)$ and infinite dimensional representations of $\laosp(1|2k)$.
In particular, in $\cP((\mR^m)^k)\otimes\mS$, the isotypic components of $\Spin(m)$ give explicit realizations of lowest weight modules of $\laosp(1|2k)$ with lowest weights $(a_1+(m/2),\ldots,a_k+(m/2))$ for integers $a_1\geq \cdots\geq a_k\geq 0$, 
see Theorem \ref{t_isod}. In this connection, we mention that there are not so many known constructions of such modules, see e.g.\ the paraboson Fock space \cite{LSJ}, and for a~classification, we refer to \cite{DZ,DS}.  

In the paper, after preliminaries and notation in Section 2, the Fischer 
decomposition in the scalar case is reviewed in Section  3. A proof of the main result, that is, the monogenic Fischer decomposition in the stable range, is then contained in Section  4.
In Section 5, we describe the structure of isotypic components for $\Spin(m)$ in the space of spinor valued polynomials.
 
 \section{Preliminaries and notations}
 \subsection{Clifford algebra and spinors}
 Consider the Euclidean space $\mR^m$ with a fixed
 negative definite quadratic form $B.$ 
 The corresponding (universal) Clifford algebra is denoted by
 $\mR_{0,m},$   its complexification by 
 $\cC_m=\mR_{0,m}\otimes\mC.$
  If $(e_1,\ldots,e_m)$ is an orthonormal basis
 for $\mR^m,$ its elements satisfy the relations
 $$
 e_ie_j+e_je_i=-2\delta_{ij},\ i\not=j,\ i,j=1,\ldots,m.
 $$
  Let $a\to\bar{a}$ denote the main antiinvolution on $\mR^m$
 characterized by the properties $\overline{e_i}=-e_i,\;i=1,\ldots,m$
 and $\overline{ab}=\bar{b}\bar{a},$ $a,b\in\mR_{0,m}.$
 The extension of the main antiinvolution to $\cC_m$ is defined by
 $$
 \overline{a\otimes\alpha}= 
 \overline{a}\otimes \overline{\alpha},
 \ a\in\mR_{0,m},\alpha\in\mC,
 $$
 where $\overline{\alpha}$ denotes the complex conjugation.

The Clifford algebra $\cC_m$ is $\mZ_2$-graded, it decomposes as
$\cC_m=\cC_m^+\oplus\cC_m^-$ into even and odd parts.
For $m$ odd, the even part $\cC^+_m$ is (isomorphic to) a matrix algebra over $\mC.$
 For $m$ even, the algebra $\cC_m$ is (isomorphic to) a matrix algebra over $\mC,$ and $\cC_m^+$ is a sum of two matrix algebras. 

Spinor valued fields considered in the paper have values in the
space $\mS$ defined as follows.
For  odd $m=2n+1$, we denote by $\mS$ the unique irreducible module
for $\cC_m^+,$ 
while  for even $m=2n$, we denote by $\mS$ the unique irreducible module
for $\cC_m.$
In both cases, the dimension of $\mS$  is equal to $2^n.$

The vector space $\mR^m$ is embedded into $\cC_m$ by
$$x:=(x_1,\ldots, x_m)\mapsto \ux:=\sum_{i=1}^m e_ix_i,$$ the Dirac operator
is denoted by 
$\upx:=\sum_{i=1}^m e_i\partial_{x_i}.$
See e.g.\ \cite{DSS} for a~detailed account.
 
 \subsection{Fischer inner product} 
We identify $(\mR^m)^k$ with the vector space of all $m\times k$ real matrices.
The columns $x_1,\ldots,x_k$ of a~matrix $x\in(\mR^m)^k$ are hence vectors in the Euclidean
space $\mR^m$, that is, $x_j=(x_{1j},\ldots,x_{mj})\in\mR^m$.
Denote by $\cP=\cP((\mR^m)^k)$ the space of all complex valued polynomials on 
$(\mR^m)^k$.
 Then a~polynomial $f\in\cP((\mR^m)^k)$ has a form
 $f=\sum_{\alpha} c_\alpha x^\alpha,$ where $\alpha=(\alpha_{ij})\in\mN_0^{m\times k}$ is a (matrix) multiindex,  
 $x^\alpha=\Pi_{ij}(x_{ij})^{\alpha_{ij}}$ and $c_\alpha\in\mC$ are non-zero for a~finite number of $\alpha\in\mN_0^{m\times k}$. 
 The standard  Fischer inner product   on the space $\cP((\mR^m)^k)$ 
 is given for $f=\sum_{\alpha} c_\alpha x^\alpha$ and $g=\sum_{\alpha} d_\alpha x^\alpha$
 by       $$\langle f,g\rangle=\sum_\alpha \alpha!\;\overline{c_{\alpha}}d_\alpha,$$
 where $\alpha!=\Pi_{ij}(\alpha_{ij}!).$
  It is a Hermitian (positive definite) inner product. It can also be written using an integral formula (\cite{G}, Lemma 9.1).
 
 The Fischer inner product  can be extended to polynomials with values
 in the Clifford algebra $\cC_m$ by the   formula
 $$
 \langle f,g\rangle=\sum_\alpha \alpha!\;[  \overline{c_{\alpha}}d_\alpha]_0
 $$
 where, for a Clifford number $c\in\cC_m$, $\overline{c}$ is the main antiinvolution of $c$
 and $[c]_0$ is its real part.
 
 The spinor representation $\mS$ can be realized as a suitable left ideal
 in $\cC_m,$ so the Fischer inner product is well defined also for
 spinor valued fields. See \cite{DSS}.

\subsection{A realization of $\lasp(2k)$}\label{sp2k} For an account of representation theory and Howe duality, we refer e.g. to \cite{GW}.
In the description  of the Howe dual pair $(\SO(m),\lasp(2k))$
the second partner is realized inside the Weyl algebra
 $\cD\cW$ of differential operators with polynomial coefficients
 as follows.
The group $\SO(m) $ acts on the space $(\mR^m)^k$ by the left matrix multiplication, hence it induces the  action on the space $\cP=\cP((\mR^m)^k)$ of all complex valued polynomials on 
$(\mR^m)^k$
by
$$
[g\cdot P](x):=P(g^t\,x),\;g\in\SO(m),\; P\in\cP,\; x\in(\mR^m)^k.
$$
We introduce the following differential operators on the space
$\cP:$
$$
r^2_{ij}:=\sum_{k=1}^m x_{ki} x_{kj},\ 
\Delta_{ij}:= \sum_{k=1}^m \partial_{x_{ki}} \partial_{x_{kj}},\ 
E_{ij}:=\sum_{k=1}^m x_{ki} \partial_{x_{kj}}
$$
and $h_{ij}:=E_{ij}+\frac{m}{2}\delta_{ij}$ for $i,j=1,\ldots,k$.
All these differential operators are elements of the Weyl algebra
$\cD\cW$ generated by the partial derivatives $\partial_{x_{kj}}$ 
and by the operators of multiplication by $x_{kj}$ 
acting on the space $\cP$ for $k=1,\ldots,m$, $j=1,\ldots,k$.

The action of the Lie algebra $\lat\simeq\lagl(k)$ on $\cP$ given by
$$
\lat:=\lspan\{h_{ij}|\ i,j=1,\ldots k\}
$$
extends to the action of the Lie algebra
$\lag'\simeq \lasp(2k),$
where 
$\lag'=\lap_-
\oplus \lat
\oplus \lap_+,
$
and
$$
\lap_+:=\lspan\{r^2_{ij}|\ 1\leq i\leq j\leq k\},\;
\lap_-:=\lspan\{\Delta_{ij}|\ 1\leq i\leq j\leq k\}.
$$
All operators in $\lasp(2k)$ are $\SO(m)$-invariant differential operators.
The Lie algebra $\lat$ splits moreover as
$$\lat=\lat_-\oplus\lat_0\oplus\lat_+\text{\ \ with\ \ }\lat_0=\lspan\{h_{ii}|\ 1\leq i\leq k \},$$
$$\lat_+=\lspan\{h_{ji}|\ 1\leq i<j\leq k \},\ \ \lat_-=\lspan\{h_{ij}|\ 1\leq i<j\leq k \}.$$

\subsection{Lie superalgebra $\laosp(1|2k)$}\label{ss_osp}

The Lie algebra $\lasp(2k)$ is the even part of the Lie superalgebra
$\laosp(1|2k).$ There is a realization of this superalgebra 
by differential operators acting on the space of\
 spinor valued polynomials
$\cP((\mR^m)^k)\otimes  \mS.$

The even part $\lasp(2k)$ of $\laosp(1|2k)$ is realized as in Subsection \ref{sp2k} above and its odd
part is $\laf_+\oplus\laf_-,$
where
$$
\laf_+:=\lspan\{\ux_i|\ i=1,\ldots,k \},
\; \laf_-:=\lspan\{{\upx}_i|\ i=1,\ldots,k \}.
$$
So we have
$\laosp(1|2k)=\laf_-\oplus\lap_-\oplus\lat\oplus\lap_+\oplus\laf_+$.

\subsection{Representations of $\Spin(m)$}
We always assume that $m>2$, and $m=2n$ or $m=2n+1$. Finite dimensional irreducible representations of the group
$\Spin(m)$, the double cover of $SO(m)$, are denoted in the paper by $E_\lambda, $ where
$\lambda=(\lambda_1,\ldots,\lambda_n)$ is the highest weight
satisfying usual dominant conditions depending on parity of $m.$
It is given by 
$\lambda_1\geq\ldots\geq\lambda_{n-1}\geq |\lambda_n|$ for
$m=2n,$ resp. $\lambda_1\geq\ldots\geq \lambda_n\geq 0$
for $m=2n+1.$
A representation $E_\lambda$ factorizes to a representation
of $\SO(m)$ if and only if its components are integral. In case that the components
$\lambda_i$ are all elements of $\mZ+\frac{1}{2},$ it is a genuine
module for $\Spin(m).$

The spinor representation 
$\mS$ is irreducible for $m$ odd and its highest weight
is $\mu=(\frac{1}{2},\ldots,\frac{1}{2}).$
For $m$ even, $\mS=\mS^+\oplus
\mS^-$ where $\mS^{\pm}$ are irreducible submodules with the highest weights
$\mu_\pm=(\frac{1}{2},\ldots,\frac{1}{2},\pm\frac{1}{2}),$ respectively.

All irreducible $\Spin(m)$-modules can be realized using the basic function spaces studied in Clifford analysis, as shown in \cite{CLS}.
To do this, we define the space $\cH$ of harmonic polynomials on $(\mR^m)^k$ as
$$\cH:=\Ker\lap_-=\{P\in\cP|\ LP=0\text{\ for all\ }L\in\lap_-\},$$  and 
the space of simplicial harmonics
$$\cH^S:=\cH\cap\Ker \lat_-$$ 
where $\lap_-$ and $\lat_-$ are given as in Subsection \ref{ss_osp}. 
For $\ell\in\mN_0$, denote by $\cP_{\ell}$ the space of polynomials of $\cP$ homogeneous of total degree $\ell$ and, for a~multiindex $a=(a_1,\dots,a_k)\in\mN_0^k$, $\cP_a$ stands for the subspace of polynomials of $\cP$ homogeneous in the variable $x_i$ of degree $a_i$ for each $i=1,\ldots,k$.  We use an analogous notation for other spaces of polynomials, e.g., $\cH_{\ell}$ and $\cH^S_a$. 
For a partition $a\in\mN_0^k$ (i.e., $a_1\geq a_2\geq \cdots\geq a_k$), we know that the space  $\cH^S_a$ of simplicial harmonic polynomials forms an irreducible $SO(m)$-module with the highest weight $a.$
Irreducible representations with half integral weights can be
realized in a similar way using suitable spaces of monogenic polynomials. More details are presented below in Subsection \ref{s_monog}.

The following result is needed in the proof of the main theorem.
\begin{lemma}\label{Klimyk}
	Suppose that	$2k\leq m$, and $m=2n$ or $m=2n+1.$
	Let $\Pi(\mS)$ be the list of all $2^n$ weights of the module $\mS.$ 	
	Then the tensor product $E_\lambda\otimes\mS$ decomposes with multiplicity one as
	$$
	E_\lambda\otimes\mS\simeq \bigoplus_{\nu\in A}E_\nu,
	$$
	where  $A$ is the set of all dominant weights $\nu$ for $\Spin(m)$ of the form $\nu=\lambda+\alpha$
	for some $\alpha\in \Pi(\mS)$.  
	
	The number of summands is bounded by $2^k$ and is equal to $2^k,$
	if $\lambda$ lies inside the dominant Weyl chamber.

\end{lemma}

\begin{proof}
	Suppose that $\rho$ is the half sum of positive roots.	The claim follows immediately from Klimyk's formula 
	\cite[Ex.\ 25.41, p.\ 428]{FH}, because
	all weights $\alpha\in\Pi(\mS)$ have multiplicity one and
	the sum $\lambda+\alpha+\rho$ is either inside or on the wall of the
	dominant Weyl chamber.

	The first $k$ components of $\alpha\in\Pi(\mS)$ are equal
	to $\pm\frac{1}{2}$ and there are at most $2^k$ such possibilities,
	and if $\lambda+\alpha$ is dominant, all components $\alpha_i,
	i=k+1,\ldots,\ell-1$ have the sign plus.
	For a generic $\lambda$ all sums $\lambda+\alpha$ are dominant.
\end{proof}

\subsection{Representations of $\lagl(k)$}
Finite dimensional irreducible representations of the group
$\GL(k)$ are denoted in the paper by $F_\lambda, $ where
$\lambda=(\lambda_1,\ldots,\lambda_k)\in\mZ^k$ is the highest weight
satisfying usual dominant condition
$\lambda_1\geq\ldots\geq \lambda_k.$
For the Lie algebra $\lagl(k)$ of $k\times k$ matrices, there are more finite dimensional irreducible representations.
In fact, for each $\delta\in\mC$, we have 1-dimensional representation $F_{\delta 1_k}$ of $\lagl(k)$ given by
$$[g\cdot z]=\delta\trace(g)z,\ g\in\lagl(k), z\in\mC$$
where $\delta 1_k=(\delta,\ldots,\delta)$ ($k$ numbers) and $\trace(g)$ is the trace of a~matrix $g\in\lagl(k)$.
Then each finite dimensional irreducible representation of $\lagl(k)$ is of the form
$$F_{\lambda+\delta 1_k}=F_{\lambda}\otimes F_{\delta 1_k}$$ for some $\lambda=(\lambda_1,\ldots,\lambda_k)\in\mZ^k$, $\lambda_1\geq\ldots\geq \lambda_k$ and $\delta\in\mC$. In particular, we have $\dim F_{\lambda+\delta 1_k}=\dim F_{\lambda}$.

\section{Harmonic Fischer decomposition} 
In this section, we present a review of results from classical invariant theory and the theory of Howe dual pairs needed for the proof of the Fischer decomposition for spinor valued polynomials. Results presented in this section
are taken from \cite{Ho1,G,HTW}.

\subsection{Separation of variables}
The classical invariant theory describes  the set
 $\cI:=\cP^{\SO(m)}$ of invariant polynomials with respect
 to the action of the group $\SO(m).$
The theorem below 
 (called by tradition 'separation of variables') describes polynomials in the space  $\cP$ as products of invariant and harmonic polynomials. Recall that $\cH=\Ker\lap_-$ is the space of harmonic polynomials on $(\mR^m)^k$, that is, those annihilated by all the
 	differential operators $\Delta_{ij}$.
 
 \begin{thm}[{\cite[Theorem 3.3]{HTW}}] \label{t_hfd*}\	
  	
\noindent 
(i) The space $\cI=\cP^{\SO(m)}$ of invariant polynomials under the action of the group $SO(m)$ is the polynomial algebra $\mC[r^2_{ij}] $, and it is also isomorphic to the symmetric algebra $\cS(\lap_+)$ over the space $\lap_+$.
 
\noindent 
 (ii)	
 The linear map from $\cI\otimes\cH$ onto $\cP$ 
$$I\otimes H\mapsto IH,$$ given by the multiplication of polynomials $I\in\cI$ and $H\in\cH$,
 is an isomorphism if the stable range condition  $m\geq 2k$ holds. 
 \end{thm}


Obviously, in the stable range, we can reformulate Theorem \ref{t_hfd*}  as follows.

\begin{thm}\label{t_hfd} 
	If $m\geq 2k$, we have 
	\begin{equation}\label{e_hfd}
	\cP((\mR^m)^k)=\bigoplus_{\{n_{ij}\}} \Big(\prod_{1\leq i\leq j\leq k}r_{ij}^{2n_{ij}}\Big) \cH
	\end{equation}
	where the sum is taken over all sequences $\{n_{ij}|\ 1\leq i\leq j\leq k\}$ of numbers of $\mN_0$.
\end{thm}

\subsection{Decomposition of spherical harmonics}
The space $\cH$ of harmonic polynomials is invariant under the action
of the product $\SO(m)\times \lagl(k),$  and it decomposes into irreducible parts with multiplicity one as follows.
Here the Lie algebra $\lagl(k)$ is realized as $\lat$, see Subsection \ref{sp2k}.

\begin{thm}[{\cite[Theorems 10.1, 10.4]{G}}]\label{t_isodh}
The space $\cH$ of harmonic polynomials has a multiplicity free decomposition under the action of
$\SO(m)\times\lagl(k)$ 
$$
\cH=\bigoplus_a  \cH_{(a)}\text{\ \ with\ \ } \cH_{(a)}\simeq \cH^S_{a}\otimes F_{a+\frac{m}{2}1_k},
$$
where the sum is taken over all partitions  $a\in\mN_0^k$.
Individual summands $\cH_{(a)}$ are at the same time isotypic components
of the $\SO(m)$ action and the isotypic components under
the action of $\lat\simeq\lagl(k)$ in the space $\cH.$ 
\end{thm}

 \subsection{The Howe duality $\SO(m)\times\lasp(2k)$}

It is well-known that, under the joint action of 
$\SO(m)$ and $\lasp(2k)$, the space
$\cP((\mR^m)^k)$ of scalar valued polynomials has a multiplicity free decomposition. Indeed, we have

\begin{thm}[{\cite[Theorems 10.1, 10.4]{G}}]
	Assume  $m\geq 2k$.
	Then the space $\cP((\mR^m)^k)$ decomposes under the action of 
	the dual pair  $\SO(m)\times\lasp(2k)$ as
	$$
	\cP((\mR^m)^k)\simeq\bigoplus_a \cH^S_a\otimes {L}_{a+\frac m2 1_k},
	$$ 
	where the sum is taken over all partitions $a\in \mN_0^k$ and $L_{\lambda}$ is an irreducible lowest weight module with lowest weight $\lambda$ for $\lasp(2k)$.	  
\end{thm}

\section{Monogenic Fischer decomposition} 

In this section, we prove Theorem \ref{t_mfd}, 
the Fischer decomposition for spinor valued polynomials in $k$ vector variables of $\mR^m$.
In what follows, we consider only the stable range case
$m\geq 2k.$

\subsection{Radial algebra} 

 The vector variables $\ux_j=\sum_{k=1}^{m}e_kx_{kj}$ 
 are elements in the algebra $\cA:=\cP((\mR^m)^k)\otimes\End(\mS)$ of $\End(\mS)$-valued polynomials. They are invariant
 with respect to the action of the group $\Spin(m).$
Let $\cJ$
be the subalgebra generated by $\{\ux_1,\ldots,\ux_k\}$ in $\cA.$
By \cite{Som1}, $\cJ$ is a realization
of the  (abstract) radial algebra $\cR(\ux_1,\ldots,\ux_k)$ in the vector variables $\ux_1,\ldots,\ux_k$.
It is easy to see that
$$\cJ\simeq\cS(\lap_+)\otimes\largewedge(\laf_+)$$ 
where 
$\cS(\lap_+)$ is the symmetric algebra over the space $\lap_+$ and
$\largewedge(\laf_+)$ is the exterior algebra over $\laf_+$. Actually, $\cJ$ may be also viewed as the symmetric superalgebra over the superspace $V=V_0\oplus V_1$ with $V_0=\lap_+$ and $V_1=\laf_+$.


\subsection{Decomposition of  spherical monogenics}\label{s_monog}
Before proving Theorem \ref{t_mfd} we describe an irreducible decomposition of monogenic polynomials with respect to the group $\Spin(m)$ in Theorem \ref{t_idm} below. 
Recall that a~polynomial $P:(\mR^m)^k\to\mS$ is called monogenic (i.e., $P\in\cM$) if it satisfies the Dirac equations 
$$\partial_{\ux_1} P=0,\ldots, \partial_{\ux_k} P=0.$$
We say that such a~$P$ is simplicial monogenic if it holds in addition that
$$h_{ij}\; P=0$$ for each $1\leq i<j\leq k$. 
Here the operators $h_{ij}$ are defined in Subsection \ref{sp2k}.
Then, for the space $\cM^S$ of simplicial monogenics, we have
	$$\cM^S=\cM\cap\Ker \lat_-.$$ 
It is easy to see that 
$\cM^S$ decomposes by homogeneity as
\begin{equation}\label{e_dsm}
	\cM^S:=\bigoplus_{a}\cM^S_a
\end{equation}
	where the sum is taken over all partitions $a\in \mN_0^k$. In addition,
the space $\cM^S_a$ is an irreducible $\Spin(m)$-module with the highest weight 
$a' =
(a_1+\frac 12,\ldots,a_k+\frac 12,\frac 12,\ldots, \frac 12)$ ($n$\ numbers) in odd dimension $m=2n+1,$ while in even dimension $m=2n,$ it decomposes
into two irreducible components
$\cM^S_a=   \cM^{S,+}_a\oplus \cM^{S,-}_a$
with the highest weights
$a'_{\pm} =
(a_1+\frac 12,\ldots,a_k+\frac 12,\frac 12,\ldots, \frac 12,\pm\frac 12)$ ($n$\ numbers). 
Here $\cM^{S,\pm}_a=\cM^{S}_a\cap(\cP\otimes\mS^{\pm})$.
See \cite{CLS} for details.

\begin{lemma}\label{l_idhs}
For a~partition $a\in\mN_0^k$, the following $Spin(m)$-modules are isomorphic 
$$\cH^S_a\otimes\mS\simeq\bigoplus_J \cM^S_{a-\epsilon(J)}$$
where the direct sum is taken over all sets $J\subset\{1,2,\ldots,k\}$, and 
$\epsilon(J)=(\epsilon_1,\ldots,\epsilon_k)$ with
$\epsilon_j=1$ for $j\in J$ and
$\epsilon_j=0$ for $j\not\in J$. Here $\cM^S_{b}=0$ unless $b\in\mN_0^k$ is a~partition.
\end{lemma}

\begin{proof} This follows directly from Klimyk's formula, see Lemma \ref{Klimyk}. 
In the even dimension $m$, we use in addition the decompositions
$\mS=\mS^+\oplus\mS^-$ and 
$\cM^S_a=   \cM^{S,+}_a\oplus \cM^{S,-}_a$.\end{proof}

\begin{lemma}\label{l_odm} Let $\ell\in\mN_0$.
(i) Then we have
$$\cM_{\ell}=\cM^S_{\ell}\oplus\sum_{1\leq i<j\leq k} h_{ji}\cM_{\ell}.$$
Here the sum in the second term on the right-hand side is not necessarily direct. 

\noindent
(ii) In particular, we have $\cM_{\ell}=\cU(\lat_+)\cM^S_{\ell}$
where $\cU(\lat_+)$ is the universal enveloping algebra of $\lat_+$.
\end{lemma}

\begin{proof} With respect to the Fischer inner product, we have inside $\cM_{\ell}$ that
$$(\sum_{1\leq i<j\leq k} h_{ji}\cM_{\ell})^{\perp}=\cM^S_{\ell},$$
which implies (i).

	By repeated application of the claim (i), we get (ii). 
\end{proof}

\begin{thm}\label{t_idm} 
Under the action of $\Spin(m)\times \lagl(k)$, we get a~decomposition of monogenic polynomials
\begin{equation}\label{idm}
\cM=\bigoplus_a\cM_{(a)}\text{\ \ with\ \ }\cM_{(a)}\simeq\cM^S_{a}\otimes F_{a+\frac m2 1_k}
\end{equation}
where the direct sum is taken over all partitions $a\in\mN_0^k$. The decomposition is irreducible in the odd dimension $m$. In the even dimension $m$, we have  $\cM_{(a)}=\cM^+_{(a)}\oplus\cM^-_{(a)}$ with $\cM^{\pm}_{(a)}\simeq\cM^{S,\pm}_{a}\otimes F_{a+\frac m2 1_k}$. 
\end{thm}

\begin{proof}
	
	For simplicity, we limit ourselves to the odd dimension $m$.
	The even dimensional case is proved in quite a~similar way.
	
	By Lemma \ref{l_odm} and \eqref{e_dsm}, it is clear that we have
	$$\cM=\bigoplus_a\cM_{(a)}$$
	where we put
	$\cM_{(a)}:=\cU(\lat_+)\cM^S_{a}.$ Obviously, for a~given partition $a\in\mN_0^k$,
	$\cM_{(a)}$ is a~representation of $\Spin(m)\times \lagl(k)$. Let $M_a$ be a~highest weight vector of the irreducible $\Spin(m)$-module $\cM^S_{a}$ of weight $a'$. Then, under the action of $\lat\simeq\lagl(k)$, $M_a$ is also a~singular vector in $\cM_{(a)}$ (i.e., $\lat_-\cdot M_a=0$) of weight $a+\frac m2 1_k$. Actually, it is easy to see that, under the joint action of $\Spin(m)\times \lagl(k)$, $M_a$ is a~unique (up to non-zero multiple) singular vector in $\cM_{(a)}$. In other words, $\Spin(m)\times \lagl(k)$-module $\cM_{(a)}$ is irreducible and  
$$\cM_{(a)}\simeq\cM^S_{a}\otimes F_{a+\frac m2 1_k},$$
which completes the proof.		
\end{proof}

\subsection{A proof of Theorem \ref{t_mfd}}

To prove the main result of the paper we need some auxiliary results. 

\begin{lemma}\label{l_wmfd}
	We have
	$$\cP\otimes\mS=\sum_{J,\{n_{ij}\}} \Big(\prod_{1\leq i\leq j\leq k}r_{ij}^{2n_{ij}}\Big) \ux_J  \cM$$
	where the (not necessarily direct) sum is taken over all sets $J\subset\{1,2,\ldots,k\}$ and all sequences $\{n_{ij}|\ 1\leq i\leq j\leq k\}$ of numbers of $\mN_0$.
\end{lemma}

\begin{proof} Let $\ell\in\mN_0$. With respect to the Fischer inner product, we have inside $\cP_{\ell}\otimes\mS$ that
	$$(\ux_1\cP_{\ell-1}\otimes\mS+\cdots+\ux_k\cP_{\ell-1}\otimes\mS)^{\perp}=\cM_{\ell}.$$ Therefore we get
	$$\cP\otimes\mS=\cM\oplus(\ux_1\cP\otimes\mS+\cdots+\ux_k\cP\otimes\mS).$$
	By induction, we have easily that
	$$\cP\otimes\mS=\sum w\cM$$
	where the sum is taken over all finite products $w$ of the variables $\ux_1,\ux_2,\ldots,\ux_k$.
	To finish the proof we use the relations $\ux_j\ux_i=-\ux_i\ux_j+2r^2_{ij}$.
\end{proof}

To complete the proof of the main result, we need to decompose spinor valued spherical harmonics into monogenic polynomials. 
To do this, for $\ell\in\mN_0$, it is easy to see  the decomposition
\begin{equation}\label{e_HarmDecomp}
\cP_{\ell}\otimes\mS=(\cH_{\ell}\otimes\mS)\oplus\sum_{1\leq i\leq j\leq k} r^2_{ij}\;(\cP_{\ell-2}\otimes\mS).
\end{equation}
Indeed, with respect to the Fischer inner product, we have inside $\cP_{\ell}\otimes\mS$ that
	$$\Big(\sum_{1\leq i\leq j\leq k} r^2_{ij}\;(\cP_{\ell-2}\otimes\mS)\Big)^{\perp}=\cH_{\ell}\otimes\mS.$$
The projection from $\cP_{\ell}\otimes\mS$ onto $\cH_{\ell}\otimes\mS$ corresponding to the decomposition \eqref{e_HarmDecomp}	
is denoted by $\pi$, and we call it the harmonic projection.

\begin{thm}\label{t_idh} 
	For $\ell\in\mN_0$, we have
	$$\cH_{\ell}\otimes\mS=\bigoplus_J \pi(\ux_J\cM_{\ell-|J|})$$
	where the direct sum is taken over all sets $J\subset\{1,2,\ldots,k\}$ and
	$|J|$ is the number of elements of $J$. 
	Here $\pi:\cP_{\ell}\otimes\mS\mapsto\cH_{\ell}\otimes\mS$ is the harmonic projection.
	In addition, for any $\ell\in\mN_0$, the following vector spaces are isomorphic
	$$\pi(\ux_J\cM_{\ell})\simeq \ux_J\cM_{\ell}\simeq\cM_{\ell}.$$
\end{thm}

\begin{proof}
	By Lemma \ref{l_wmfd} and \eqref{e_HarmDecomp}, 
	for a~given $\ell\in\mN_0$, we have
	$$\cH_{\ell}\otimes\mS=\sum_J \pi(\ux_J\cM_{\ell-|J|})$$
	where the sum is taken over all sets $J\subset\{1,2,\ldots,k\}$.
	To show that the sum is actually direct it is sufficient to prove that the $\Spin(m)$-modules $M:=\cH_{\ell}\otimes\mS$ and
	$$
	N:=\bigoplus_J \cM_{\ell-|J|}
	$$
	where the direct sum is taken over all sets $J\subset\{1,2,\ldots,k\}$ are isomorphic.
	To do this, we show that multiplicities of each submodule $\cM^S_a$ in the modules $M$ and $N$ are the same.
	Indeed, let $j=1,2,\ldots,k$ and $a\in\mN_0^k$ be a~partition such that $\ell-j=|a|$ where $|a|=a_1+a_2+\cdots+a_k$.
	For $\Spin(m)$-modules $A,B$, denote  the multiplicity of $A$ in $B$ by $[A:B]$.  
	Then, by Theorem \ref{t_idm}, we have   
	\begin{equation}\label{multN}
	[\cM^S_a:N]=\sum_{J:|J|=j} [\cM^S_a:\cM_{\ell-|J|}]={k \choose j}\dim F_{a} 
	\end{equation}
	where the sum in the middle expression is taken over all subsets $J\subset\{1,2,\ldots,k\}$ with $j$ elements.
	Here we use the fact that $\dim F_{a+\frac m2 1_k}=\dim F_{a}$.
	On the other hand, using Lemma \ref{l_idhs}, we get
	\begin{equation}\label{multM}
	[\cM^S_a:M]=\sum_{J:|J|=j} \dim F_{a+\epsilon(J)}
	\end{equation}
	because $[\cM^S_a:\cH^S_{a+\epsilon(J)}\otimes\mS]=1$ and $[\cH^S_{a+\epsilon(J)}\otimes\mS:\cH_{\ell}\otimes\mS]=\dim F_{a+\epsilon(J)}$. 
	Here $F_b=0$ and $\cH^S_b=0$ unless $b\in\mN_0^k$ is a~partition.
	Hence to prove the equality $[\cM^S_a:M]=[\cM^S_a:N]$ we need to show that 
	$${k \choose j}\dim F_{a}=\sum_{J:|J|=j} \dim F_{a+\epsilon(J)}.$$ 
	But this follows directly from Pieri's rule for $\GL(k)$-modules
	$$\largewedge^j(\mC^k)\otimes F_a\simeq\bigoplus_{J:|J|=j} F_{a+\epsilon(J)} $$
	where $\largewedge^j(\mC^k)$ is the $j$-th antisymmetric power of the defining representation $\mC^k$ for $\GL(k)$ and 
	$\dim\largewedge^j(\mC^k)={k \choose j}$, see \cite{HL}. 
	
	In addition, for each $\ell\in\mN_0$, we have proved that $\pi(\ux_J\cM_{\ell})\simeq\cM_{\ell}.$ Finally, the fact that 
	$\ux_J\cM_{\ell}\simeq\cM_{\ell}$ is trivial.
\end{proof}

\begin{remark} The use of Pieri's rule for $\GL(k)$-modules in the proof of Theorem \ref{t_idh} is not a~chance. In fact, from the proof, it is not difficult to see that
	$$\cH\otimes\mS\simeq\largewedge(\ux_1,\ux_2,\ldots,\ux_k)\otimes\cM$$ 
	as $\Spin(m)\times \lagl(k)$-modules. Here the action of $\lagl(k)$ is trivial on $\mS$ and  the action of $\Spin(m)$ is trivial on $\largewedge(\ux_1,\ux_2,\ldots,\ux_k)$. 
\end{remark}

\begin{proof}[Proof of Theorem \ref{t_mfd}]
	Let $\ell\in\mN_0$. By Lemma \ref{l_wmfd}, we know that 
	\begin{equation}\label{e_mfd1}
	\cP_{\ell}\otimes\mS=\sum \Big(\prod_{1\leq i\leq j\leq k}r_{ij}^{2n_{ij}}\Big) \ux_J  \cM_t
	\end{equation}
	where the sum is taken over all $J$ and $\{n_{ij}\}$ such that $\ell=t+|J|+2\sum_{1\leq i\leq j\leq k} n_{ij}$.
	On the other hand, by Theorems \ref{t_hfd} and \ref{t_idh}, we have
	\begin{equation}\label{e_mfd2}
	\cP_{\ell}\otimes\mS=\bigoplus \Big(\prod_{1\leq i\leq j\leq k}r_{ij}^{2n_{ij}}\Big) \pi(\ux_J  \cM_t)\simeq \bigoplus\Big(\prod_{1\leq i\leq j\leq k}r_{ij}^{2n_{ij}}\Big) \ux_J  \cM_t
	\end{equation} where the direct sums are taken over the same set of parameters $J$ and $\{n_{ij}\}$ as in \eqref{e_mfd1}. 
	Finally, we get \eqref{e_mfd} from \eqref{e_mfd1} and \eqref{e_mfd2}.
\end{proof}

\section{Isotypic components} 

In this section, we describe the structure of isotypic components for $\Spin(m)$ in the space of spinor valued polynomials.
To do this, recall the decomposition of $\laosp(1|2k)$ from Subsection \ref{ss_osp}
$$\laosp(1|2k)=\laf_-\oplus\lap_-\oplus\lat\oplus\lap_+\oplus\laf_+.$$ Given a~finite dimensional irreducible modul $F_{\lambda}$ for $\lat\simeq\lagl(k)$, we extend the action of $\lan_{-}:=\laf_-\oplus\lap_-$ on $F_{\lambda}$ trivially and define the generalized Verma module $V_{\lambda}$ for $\laosp(1|2k)$ as the induced module
$$V_{\lambda}=\Ind_{\lan_-\oplus\lat}^{\laosp(1|2k)} F_{\lambda}.$$

\begin{thm}\label{t_isod} 
	If $m\geq 2k$, then we have
\begin{equation}\label{e_isod}
	\cP((\mR^m)^k)\otimes\mS\simeq\bigoplus_a \cM^S_{a}\otimes V_{a+\frac m2 1_k}
\end{equation}
where the direct sum is taken over all partitions $a\in\mN_0^k$. 
\end{thm}

\begin{proof} The decomposition \eqref{e_isod} follows directly from Theorems \ref{t_mfd} and \ref{t_idm}  and PBW theorem for Lie superalgebras (see \cite[Theorem 1.36, p. 31]{CW}). 
\end{proof}

We expect that $\laosp(1|2k)$-modules $V_{a+\frac m2 1_k}$ that occur in the decomposition \eqref{e_isod} are irreducible. But this question seems to remain open. 



\def\bibname{Bibliography}

\end{document}